\documentclass[12pt]{article}
\usepackage{amsmath}
\usepackage{graphicx}
\usepackage{hyperref}
\usepackage[latin1]{inputenc}
\usepackage{mathtools}
\usepackage{amsfonts}
\usepackage{amssymb}
\usepackage[T1]{fontenc}
\usepackage{amsthm}
\usepackage{fullpage}

\theoremstyle{definition}
\newtheorem{definition}{Definition}[section]
 
\theoremstyle{remark}

\theoremstyle{theorem}
\newtheorem{theorem}{Theorem}[section]

\newtheorem{lemma}[theorem]{Lemma}

\numberwithin{equation}{section}

\newcommand{\F}{\mathbb{F}_q}
\newcommand{\Fm}{\mathbb{F}_{q^m}}
\newcommand{\N}{\mathfrak{N}}

\title{The existence of primitive normal elements of quadratic forms over finite fields }
\author{ Himangshu Hazarika$^1$, Dhiren Kumar Basnet$^{2*}$ and Stephen D Cohen$^3$\\$^{1,2}$Department of Mathematical Sciences\\ Tezpur University, Assam, India\\
$^3$ 6 Bracken Road,
Portlethen, 
Aberdeen AB12 4TA, Scotland, UK }
\date{}

\begin{document}
\maketitle

\begin{abstract}
 For $q=3^r$ ($r>0$), denote by $\mathbb{F}_q$ the finite field of order $q$ and for a positive integer $m\geq2$, let $\mathbb{F}_{q^m}$ be its extension field of degree $m$. We establish a sufficient condition for  existence of a primitive normal element $\alpha$ such that $f(\alpha)$ is a primitive element, where $f(x)= ax^2+bx+c$, with $a,b,c\in \mathbb{F}_{q^m}$ satisfying $b^2\neq ac$ in $\mathbb{F}_{q^m}$ except for at most 9 exceptional pairs $(q,m)$. 
 \end{abstract}

\section{Introduction}
Given a prime power  $q$  and an integer $m\geq2$ , we denote the finite field of order $q$ by $\mathbb{F}_q$ and its extension field of degree $m$ by $\mathbb{F}_{q^m}$. A generator of the (cyclic)  multiplicative group $\mathbb{F}^*_{q^m}$ is called a {\em primitive} element. Further, an element $\alpha\in\mathbb{F}_{q^n}$ for which the set $\{\alpha,\alpha^q, \alpha^{q^2},\ldots, \alpha^{q^{m-1}}\}$ is a $\mathbb{F}_q$ basis of $\mathbb{F}_{q^m}$ is called a $normal$ element or a $free$ element; such a  basis is called a $normal$ basis.

 For the existence of both primitive and free elements we refer to \cite{10}. The simultaneous occurrence of primitive and free elements in $\mathbb{F}_{q^m}$ is given by the following theorems.
\begin{theorem}{\textbf{ \em (Primitive normal basis theorem)}}. In the finite field $\mathbb{F}_{q^m}$, there always exists some element which is simultaneously primitive and free.
\end{theorem}

 This result was first proved by Lenstra and Schoof in \cite{2}. Later on by using a sieving technique, Cohen and Huczynska \cite{3} provided a computer-free proof.   
   
\begin{theorem}{\textbf{\em (Strong primitive normal basis theorem $\cite{4}$)}} In the finite field $\mathbb{F}_{q^m}$, there exists some element $\alpha$ such that both $\alpha$ and $\alpha^{-1}$ are primitive normal, with exceptional pairs for $(q,m)$ are $ (2,3),(2,4),(3,4),(4,3)$ and $(5,4)$. 

\end{theorem}
This result was first proved by Tian and Qi in \cite{5} for $m\geq 32$. Later on Cohen and Huczynska \cite{4} completed the proof up to the above form by using a  sieving technique. 

The existence of a primitive element $\alpha \in \F$ for which $f(\alpha)$ is also primitive for an arbitrary quadratic in $\F[x]$ has been settled in \cite{1}.
\begin{theorem}[\bf{\cite{1}}]\label{T1}
For all $q>211$, there always exists an element $\alpha\in \mathbb{F}_{q^m}$ such that $\alpha$ and $f(\alpha)$ are both primitive, where $f(x)=ax^2+bx+c$ with $b^2-4ac\neq0$.
\end{theorem}

In this paper we consider an extension of Lemma \ref{T1} by posing the existence question for primitive elements $\alpha$ of $\Fm$  that  are normal over $\F$ and for which $f(\alpha)$ is also primitive for an arbitrary  quadratic polynomia (with distinct roots)  $f(x) \in \Fm[x]$.  

We apply the Lenstra-Schoof method \cite{2} by using character sums. But for more accurate results, we use the sieving technique provided by Cohen and Huczynska \cite{3, 4}.

Finally, we apply the existence conditions to fields in which $q$ is a power of the prime $3$.

\section{Preliminaries} 
The additive group of $\mathbb{F}_{q^m}$ is a $\mathbb{F}_{q}[x]$-module under the rule $f\,o\,\alpha=\overset{n}{\underset{i=1}{\sum}} a_i\alpha^{q^i}$; for $\alpha\in \mathbb{F}_{q^m}$ and 
 $f\,= \overset{n}{\underset{i=1}{\sum}}a_ix^i\thinspace \in \mathbb{F}_{q}[x]$. For $\alpha \in \mathbb{F}_{q^m}$, the $\mathbb{F}_q$-order of $\alpha$ is the monic $\mathbb{F}_q$-divisor $g$ of $x^m-1$ of minimal degree such that $g\,o\,\alpha=0$. Hence the annihilator of $\alpha$ has unique monic generator which we define as $Order\, of\, \alpha$ and denote by Ord($\alpha$). It is clear that an element in $\mathbb{F}_{q^m}$ is free if and only if its order is exactly $x^m-1$. 
 
  Now for $\alpha\in \mathbb{F}^*_{q^m}$, the multiplicative order is denoted by ord($\alpha$) and $\alpha$ is primitive if and only if ord$(\alpha)=q^m-1$. From the definitions it is clear that $q^m-1$ and $x^m-1$ can be replaced by their radicals $q_0$ and $f_0\,:=\, x^{m_0}-1 $ respectively, where $m_0$ is such that $m=m_0p^a$, where $a$ is a non negative integer and gcd$(m_0,p)=1$.
  
   Furthermore, we use the following definitions and Lemmas in our result. 
 \begin{definition} 
 Let $G$ be a finite abelian group. Then a character $\chi$ of $G$ is a homomorphism from $G$ into the group $S^1:= \{z \in \mathbb{C} : |z| = 1\}$. The characters of $G$ form a group under multiplication called $\mathit{dual\thinspace group}$ or $\mathit{character\thinspace group}$ of $G$ which is denoted by $\widehat{G}$. It is well known that $\widehat{G}$ is isomorphic to $G$. Again the character $\chi_0$ is denoted for the trivial character of $G$ defined as $\chi_0(a)=1$ for all $a \in G$.
  \end{definition}
 
  In a finite field $\mathbb{F}_{q^m}$, there are two types of abelian groups, namely additive group $\mathbb{F}_{q^m}$ and multiplicative group $\mathbb{F}^*_{q^m}$. So, there are two types of characters of a finite field $\mathbb{F}_{q^m}$, namely $\mathit{additive \thinspace character }$ of $\mathbb{F}_{q^m}$ and $\mathit{multiplicative \thinspace character }$ of $\mathbb{F}^*_{q^m}$. Multiplicative characters are extended from $\mathbb{F}^*_{q^m}$ to $\mathbb{F}_{q^m}$ by the  rule
 \hspace{.1cm}  $\chi(0)=\begin{cases}
                         0 \,\mbox{ if}\, \chi\neq\chi_0\\
                         1 \,\mbox{ if}\, \chi=\chi_0 
                         \end{cases} $

   Since $\widehat{\mathbb{F}^*_{q^m}} \cong \mathbb{F}^*_{q^m}$, so $\widehat{\mathbb{F}^*_{q^m}}$ is cyclic and for any divisor $d$ of $q^m-1$, there are exactly $\phi(d)$ characters of order $d$ in $\widehat{\mathbb{F}^*_{q^m}}$.

   Let $e|q^m-1$, then $\alpha$ is called $e-free$ if $d|e$ and $\alpha= \beta^d$, for some $\beta\in \Fm$ implies $d=1$. Furthermore $\alpha$ is primitive if and only if $\alpha= \beta^e$, for some $\beta \in \Fm$ and $e|q^m-1$ implies $e=1$.

 For any $e|q^m-1$, Cohen and Huczynska \cite{3, 4} defined  the character function for the subset of $e$-free elements of $\mathbb{F}^*_{q^m}$ by
 $$\rho_e: \alpha\mapsto\theta(e)\underset{d|e}{\sum}(\frac{\mu(d)}{\phi(d)}\underset{\chi_d}{\sum}\chi_d(\alpha)),$$
     where $\theta(e):=\frac{\phi(e)}{e}$, $\mu$ is the M\"obius function and $\chi_d$ stands for any multiplicative character of order $d$.
   For any $e|q^m-1$, we use ``integral'' notation due to Cohen and Huczynska \cite{3, 4}, for weighted sums as follows
\begin{align*}
\underset{d|q^m-1}{\int}\chi_d:= \underset{d|q^m-1}{\sum}\frac{\mu(d)}{\phi(d)}\underset{\chi_d}{\sum}\chi_d
\end{align*}

 Then they defined the  characteristic function for the subset of $e$-free elements of $\mathbb{F}^*_{q^m}$, as follows
\begin{align*}
\rho_e: \alpha\mapsto\theta(e)\underset{d|e}{\int}\,\chi_d(\alpha)
\end{align*}
    
     Again, for any monic $\mathbb{F}_q$-divisor $g$ of $x^m-1$, a typical additive character $\psi_g$ of $\mathbb{F}_q$-order $g$ is one such that $\psi_g o g$ is the trivial character of $\mathbb{F}_{q^m}$ and $g$ is of minimal degree satisfying this property. 
    Furthermore, there are $\Phi(g)$ characters $\psi_g$, where $\Phi(g)= (\mathbb{F}_q[x]/g\mathbb{F}_q[x])^*$ is the analogue of Euler function over $\mathbb{F}_{q}[x]$. 
    
    Then the character function for the set of $g$-free elements in $\mathbb{F}_{q^m}$, for any $g|x^m-1$ is given by
    $$\kappa_g :\alpha\mapsto\Theta(g)\underset{f|g}{\sum}(\frac{\mu^\prime(f)}{\Phi(f)}\underset{\psi_f}{\sum}\psi_f(\alpha)),$$
    where $\Theta(g):= \frac{\Phi(g)}{q^{deg(g)}}$, the  sum runs over all additive characters $\psi_f$ of $\mathbb{F}_q$-order g and $\mu^\prime$ is the analogue of the M\"obius function which is defined as follows: 
      $$\mu^\prime(g)=\begin{cases}
                         (-1)^s\hspace{.2cm} \mbox{if \,g \, is \, a \, product\, of }\, s\,\mbox{ distinct\, irreducible\, monic\, polynomials} \\
                        \hspace{.1cm} 0 \qquad\mbox{otherwise}\\ 
                         \end{cases} $$ 
 
      We use the ``integral'' notation for weighted sum of additive characters as follows
\begin{align*}
\underset{f|g}{\int}\psi_f:=\underset{f|g}{\sum}\frac{\mu^\prime(f)}{\Phi(f)}\underset{\psi_f}{\sum}\psi_f
\end{align*}     
   Then the character function for the set of $g$-free elements in $\mathbb{F}_{q^m}$, for any $g|x^m-1$ is given by
\begin{align*}
\kappa_g :\alpha\mapsto\Theta(g)\underset{f|g}{\int}\,\psi_f(\alpha)
\end{align*}
From \cite{5}, we have the following about the typical additive character.
 Let $\lambda$ be the canonical additive character of $\mathbb{F}_q$. Thus for $\alpha\in \mathbb{F}_q$ this character is defined as  $\lambda(\alpha)= \exp^{2\pi iTr(\alpha)/p}$; where $Tr(\alpha)$ is absolute trace of $\alpha$ over $\mathbb{F}_p$.

 Now let $\psi_0$ be canonical additive character of $\mathbb{F}_{q^m}$; which is simply the lift of $\lambda$ to $\mathbb{F}_{q^m}$ i.e., $\psi_0(\alpha)=\lambda(Tr(\alpha)), \, \alpha\in 
\mathbb{F}_{q^m}$. Now for any $\delta\in \mathbb{F}_{q^m}$, let $\psi_\delta$ be the character defined by $\psi_\delta(\alpha)=\psi_0(\delta\alpha), \, \alpha\in \mathbb{F}_{q^m}$.
Define the subset $\Delta_g$ of $\mathbb{F}_{q^m}$ as the set of $\delta$ for which $\psi_\delta$ has $\mathbb{F}_{q}$-order $g$. So we may also write $\psi_{\delta_g}$ for $\psi_\delta$, where $\delta_g\in \Delta_g$. So with the help of this we can express any typical additive character $\psi_g$ in terms of $\psi_{\delta_g}$ and further we can express this in terms of canonical additive character $\psi_0$.  

\begin{lemma} {\bf\cite{10} }
If $\chi$ is any nontrivial character of a finite abelian group $G$ and $\alpha\in G$ any nontrivial element, then 
$$\underset{\alpha\in G}{\sum}\chi(\alpha)=0 \quad  \mbox{and} \quad \underset{\chi\in \widehat{G}}{\sum}\chi(\alpha)=0.$$ 
\end{lemma}.

\begin{lemma} \label{charbound1}
{\bf  (\cite{6})}
Consider any two nontrivial multiplicative characters $\chi_1,\chi_2$ of the finite field $\mathbb{F}_{q^m}$. Again, let $f_1(x)$ and $f_2(x)$ be two monic pairwise co-prime polynomials in $\mathbb{F}_{q^m}[x]$, such that at least one of $f_i(x)$ is of the form $g(x)^{ord(\chi_i)}$ for $i=1,2$; where $g(x)\in \mathbb{F}_{q^m}[x]$ with degree at least 1. Then 
$$\Big|\underset{\alpha\in \mathbb{F}_{q^m}}{\sum}\chi_1(f_1(\alpha))\chi_2(f_2(\alpha))\Big|\leq (n_1+n_2-1)q^{m/2},$$
 where $n_1$ and $n_2$ are the degrees of largest square free divisors of $f_1$ and $f_2$, respectively. 
\end{lemma}
\begin{lemma}\label{charbound2} {\bf (\cite{7}) } Let $f_1(x), f_2(x),\ldots, f_k(x)\in \mathbb{F}_{q^m}[x]$ be distinct irreducible polynomials over $\mathbb{F}_{q^m}$. Let $\chi_1,\chi_2, \ldots,\chi_k$ be multiplicative characters and $\psi$ be a non-trivial additive character of $\mathbb{F}_{q^m}$. Then
 $$\left | \underset{\underset{f_i(\alpha)\neq0}{\alpha\in\mathbb{F}_{q^m}}}{\sum} \chi_1(f_1(\alpha))\chi_2(f_2(\alpha))\ldots\chi_k(f_k(\alpha))\psi(\alpha)\right|\leq n\,q^{m/2}$$ 
 where $n= \overset{k}{\underset{j=1}{\sum}}deg(f_j)$. 
\end{lemma}

\section{A lower bound for $\mathfrak{N}(e_1,e_2,g)$}

  We are trying to estimate some results on the primitive normal elements $\alpha$ such that $f(\alpha)$ is primitive in $\mathbb{F}_{q^m}$. Initially we are considering $q$ as power of odd prime $p$ as $q=p^k$, where $k$ is a positive integer. Take $e_1,e_2$ such that $e_1, e_2|q^m-1$ and $g$ such that $g|x^m-1$.
Considering $\mathfrak{N}(e_1,e_2,g)$ to be the number of $\alpha\in\mathbb{F}_{q^m}$ such that $\alpha$ is both $e_1$-free and $g$-free and $f(\alpha)$ is $e_2$-free, where $f(x)=ax^2+bx+c$ and $a,b,c\in\mathbb{F}_{q^m}$, $b^2-4ac\neq 0$. 

 We use the notations $\omega(n)$ and $g_d$ to denote number of prime divisors of $n$ and the number of monic irreducible factors of $g$ over $\mathbb{F}_{q}$ respectively. For calculations we use $W(n):=2^{\omega(n)}$  and $\Omega(g):=2^{g_d}$.

\begin{theorem} Let $f(x)=ax^2+bx+c$ with $a,b,c\in \mathbb{F}_{q^m}$ and $b^2-4ac\neq0$.  Suppose $e_1, e_2$ divide $q^m-1$ and $g|x^m-1$.Then 
\begin{equation} \label {cond0}
\mathfrak{N}(e_1, e_2, g) \geq \theta(e_1) \theta(e_2)\Theta(g)q^{m/2} \left( q^{m/2}-3 W(e_1)W(e_2)\Omega(g)\right).
\end{equation}
Hence $\N(e_1,e_2,g) > 0$ whenever
\begin{equation}\label{cond1}
q^{\frac{m}{2}} > 3W(e_1)W(e_2)\Omega(g).
\end{equation}

In particular, 
$\mathfrak{N}(q^m-1,q^m-1,x^m-1)>0$ if
\begin{equation} \label{cond}
q^{m/2}>3W(q^m-1)^2\Omega(x^m-1),
\end{equation}
 i.e., this is a sufficient condition for a field $\mathbb{F}_{q^m}$ to have an element $\alpha$ which is primitive normal and $f(\alpha)$ is also primitive.
\end{theorem}

\begin{proof} At first we establish the result for $c\neq 0$.
From the definition we have, 
\begin{equation}  \label{N}
\mathfrak{N}(e_1,e_2,g)= \theta(e_1)\theta(e_2)\Theta(g)\underset{\underset{d_2|e_2}{d_1|e_1}}{\int}\,\underset{h|g}{\int}\,S(\chi_{d_1},\chi_{d_2},\psi_h),
\end{equation}
where $$S(\chi_{d_1},\chi_{d_2},\psi_h)= \underset{\alpha\in\mathbb{F}_{q^m}}{\sum}\chi_{d_1}(\alpha)\chi_{d_2}(f(\alpha))\psi_h(\alpha)$$
 and $f(x)=ax^2+bx+c$, which has no repeated roots.
 
Now, if $(\chi_{d_1},\chi_{d_2},\psi_h)\neq (\chi_0,\chi_0,\psi_0)$, then we consider the following cases.

 If $\psi_h\neq \psi_0$, then applying  Lemma \ref{charbound2}, we have $\left|S(\chi_{d_1},\chi_{d_2},\psi_h)\right|\leq 3q^{m/2}$. This result holds even if $ax^2+bx+c$ has three distinct roots in $\mathbb{F}_{q^m}$.

 When $\psi_h=\psi_0$, then applying Lemma 2.2, we have 
 $$\left|S(\chi_{d_1},\chi_{d_2},\psi_h)\right|=\left|S(\chi_{d_1},\chi_{d_2},\psi_0)\right|\leq 2q^{m/2} < 3 q^{m/2}.$$

 Additionally, if $\chi_{d_1}=\chi_{d_2}=\chi_0$, then
  $$\left| S(\chi_{d_1},\chi_{d_2},\psi_h)\right|=\left| S(\chi_{0},\chi_{0},\psi_h)\right|= \left| \underset{\alpha\in\mathbb{F}_{q^m}}{\sum}\psi_g(y\alpha)\right| =0.$$

 Hence, $\left|S(\chi_{d_1},\chi_{d_2},\psi_h)\right| \leq 3 q^{m/2}$, when $(\chi_{d_1},\chi_{d_2},\psi_h)\neq (\chi_{0},\chi_{0},\psi_0) $

Using the  above results, from (\ref{N}) and allowing for up to three zeors of $x(ax^2+bx+c)$ in $\mathbb{F}_{q^m}$,  we obtain the following inequality
\begin{eqnarray*}
\mathfrak{N}(e_1, e_2, g) &\geq &\theta(e_1) \theta(e_2)\Theta(g)\left(q^m-3-3q^{m/2}\underset{d_1,d_2\neq1}{\int}\underset{g\neq1}{\int}1\right) \\
 &=& \theta(e_1) \theta(e_2)\Theta(g) q^{m/2}\left( q^{m/2}-3q^{-m/2}-3\underset{d_1,d_2\neq1}{\int}\underset{g\neq1}{\int}1\right)\\
 &\geq &\theta(e_1) \theta(e_2)\Theta(g)q^{m/2} \left( q^{m/2}-3q^{-m/2}-3\left( W(e_1)W(e_2)\Omega(g)-1\right)\right)\\ 
 & \geq &\theta(e_1) \theta(e_2)\Theta(g)q^{m/2} \left( q^{m/2}-3 W(e_1)W(e_2)\Omega(g)\right).
 \end{eqnarray*}
 This yields (\ref{cond0}).  In particular, setting  $e_1=e_2=q^m-1$ and $g=x^m-1$ we obtain the  sufficient condition (\ref{cond}).

 \end{proof}

 We briefly consider  the case in which $c=0$. Then $f(x)= ax^2+bx= x(ax+b)$, where $a,b\in \mathbb{F}_{q^m}$ with $ab\neq 0$. This time we have
$$\mathfrak{N}(e_1,e_2,g)= \theta(e_1)\theta(e_2)\Theta(g)\underset{\underset{d_2|e_2}{d_1|e_1}}{\int}\,\underset{h|g}{\int}\,S(\chi_{d_1},\chi_{d_2},\psi_h),$$
where $$S(\chi_{d_1},\chi_{d_2},\psi_h)= \underset{\alpha\in\mathbb{F}_{q^m}}{\sum}\chi_{d_1}(\alpha)\chi_{d_2}(\alpha(a\alpha+b))\psi_h(\alpha)=  \underset{\alpha\in\mathbb{F}_{q^m}}{\sum}\chi_{d_3}(\alpha)\chi_{d_2}(a\alpha+b)\psi_h(\alpha).$$
with $\chi_{d_3}=\chi_{d_1}\chi_{d_2}$.  Now, by Lemma \ref{charbound2}.
  $$\left|S(\chi_{d_1},\chi_{d_2},\psi_h)\right|= \left|  \underset{\alpha\in\mathbb{F}_{q^m}}{\sum}\chi_{d_3}(\alpha)\chi_{d_2}(a\alpha+b)\psi_h(\alpha)\right|\leq 2q^{m/2} < 3q^{m/2}$$  and \ref{cond} and (\ref{cond0}) follow as before.

\section{The prime sieve}
For the next stage in our investigation we apply the results on primes dividing $q^m-1$ and irreducible polynolmials dividing $x^m-1$.   This was  introduced by Cohen and Huczynska in \cite{3,4}. Furthermore, Kapetanakis established the following sieving inequality in \cite{8}

\begin{lemma}\label{sieveineq}
{\bf(Sieve Inequality)}
Let $d$ be a divisor of $q^m-1$ and $p_1, p_2,\dots , p_n$ are the remaining distinct primes dividing $q^m-1$. Furthermore, let $g$ be a divisor of $x^m-1$ such that $g_1, g_2, \dots , g_k$ are the remaining distinct irreducible polynomials dividing $x^m-1$. Abbreviate  $ \mathfrak{N}(q^m-1, q^m-1, x^m-1) $ to $\mathfrak{N}$. Then
\begin{equation} \label{ineq}
 \mathfrak{N} \geq \overset{n}{\underset{i=1}{\sum}}\mathfrak{N}(p_i d, d, g)+  \overset{n}{\underset{i=1}{\sum}}\mathfrak{N}( d,p_i d, g)+  \overset{k}{\underset{i=1}{\sum}}\mathfrak{N}(d, d,g_i g)-(2n+k-1)\mathfrak{N}(d, d, g).
\end{equation}
\end{lemma}

\begin{theorem}
With all the assumptions as Lemma $ \ref{sieveineq}$,  define 
$$ \Delta := 1 - 2  \overset{n}{\underset{i=1}{\sum}}\frac{1}{p_i} - \overset{k}{\underset{i=1}{\sum}} \frac{1}{q^{{\mathrm deg}(g_i)}}$$
 and  
 $$ \Lambda := \frac{2n+k-1}{\Delta}+2.$$
 Suppose  $\Delta>0$. Then a sufficient condition such that a primitive normal element $\alpha$ for which $a\alpha^2+b\alpha+c$ is primitive over $\mathbb{F}_{q^m}$ with $b^2-4ac\neq 0$  is

\begin{equation}\label{cond3}
  q^{m/2}> 3 {W(d)}^2\Omega(g)\Lambda .
  \end{equation} 
\end{theorem}

\begin{proof}   A key step is to write (\ref{ineq}) in the equivalent form
\begin{multline}\label{altineq}
\N  \geq\sum_{i=1}^n\left( \N(p_id,d,g)-\left(1-\frac{1}{p_i}\right) \N(d,d,g) \right)+  \sum_{i=1}^n\left(\N(d,dp_i,g)-\left(1-\frac{1}{p_i}\right)\N(d,d,g)\right)\\
+\sum_{i=1}^k \left(\N(d,d,g_ig) -\left(1-\frac{1}{q^{{\mathrm deg}(g_i)}}\right)\N(d,d,g)\right) \quad  + \quad \Delta\N(d,d,g).
\end{multline}

On the right side of (\ref{altineq}), since $\Delta >0$,  we can bound the last  term below using (\ref{cond0}).   Thus
\begin{equation} \label{Cond}
 \Delta\N(d,d,g) \geq \Delta \theta^2(d)\Theta(g)q^{\frac{m}{2}}(q^{\frac{m}{2}} -3W^2(d)\Omega(g)).
 \end{equation}

Moreover, since $ \theta(p_id) =\theta(p_i)\theta(d) =\left(1-\frac{1}{p_i}\right)$,  we have from (\ref{N}),
$$ \N(p_id,d,g)-\left(1-\frac{1}{p_i}\right) \N(d,d,g) =\left(1-\frac{1}{p_i}\right)\theta^2\int_{\substack{d_1|d\\ d_2|d}}\int_{h|g}S(\chi_{p_id_1}, \chi_{d_2}, \psi_h).$$
Hence, as for (\ref{cond0}),
\begin{eqnarray}\label{diff1}
\left| \N(p_id,d,g)-\left(1-\frac{1}{p_i}\right) \N(d,d,g)\right| & \leq & 3\left(1-\frac{1}{p_i}\right)\theta^2(d)\Theta(g)\big{(}W(p_id)-W(p_i)\big{)}W(d) \nonumber\\
&=& 3\left(1-\frac{1}{p_i}\right)\theta^2(d)W^2(d) .
\end{eqnarray}
Similarly,
\begin{equation}\label{diff2}
\left| \N(d,d,g)-\left(1-\frac{1}{p_i}\right) \N(d,p_id,g)\right| \leq 3\left(1-\frac{1}{p_i}\right)\theta^2(d)\Theta(g)W^2(d)
\end{equation}
and
\begin{equation} \label{diff3}
 \left|\N(d,d,g_ig) -\left(1-\frac{1}{q^{{\mathrm deg}g_i}}\right)\N(d,d,g)\right| \leq 3 \theta^2(d) \left(1- \frac{1}{q^{\mathrm{deg}(g_i)}} \right)\Theta(g)\Omega(g).
\end{equation}

Inserting (\ref{Cond}), (\ref{diff1}), (\ref{diff2}) and (\ref{diff3}) in (\ref{altineq}) and cancelling   the common factor $\theta^2(d)\Theta(g)$, we obtain (\ref{cond3}) as a condition for $\N$ to be positive (since       $\Delta$ is positive).

 \end{proof}

\section{Existence results for fields of characteristic 3}
One could endeavour to analyse the consequences of the conditions (\ref{cond}) and (\ref{cond3}) to arbitrary pairs $(q,m)$ by extending and developing the techniques employed in \cite{5}, \cite{6} and \cite{7} but this would be a testing exercise.   Accordingly, we simply illustrate what might be possible by  drawing  on specific items in these works to deal with  finite fields of characteristic 3.   Hence, from now on we suppose $q=3^r$, where $r$ is a positive integer.   
 
 First we settle the case when $m=2$.
 \begin{lemma}\label{m2}
Let $q=3^r$.  Given any quadratic polynomial $f(x)= ax^2+bx+c \in \F[x]$, with $ a \neq 0, b^2 \neq ac$,  there exists a primitive normal element $\alpha \in \mathbb{F}_{q^2}$ such that $f(\alpha)$ is also primitive.
 \end{lemma}
 \begin{proof}
 As noted in \cite{5}, when $m=2$, a primitive element of $\Fm$ is automatically normal over $\F$.  Hence  it suffices to show that there is a primitive element of $\alpha \in \mathbb{F}_{q^2}$  such that $f (\alpha)$ is also primitive for any quadratic $f(x) \in \F[x]$ (with $b^2 \neq ac$).  By \cite{1} Theorem 1, this holds (even allowing $f(x)$ to be a quadratic in $\mathbb{F}_{q^2}[x]$) except when $q=3$.  
 
 So suppose $(q,m)=(3,2)$. and let $\beta \in \mathbb{F}_9$ satisfy $\beta^2+\beta+1$.   Then $\beta$ is a primitive element of $ \mathbb{F}_9$.   Moreover, $\{\pm \beta, \pm(\beta-1)\}$ comprises the set of primitive elements of $ \mathbb{F}_9$.  Hence we need only consider one quadratic $f$ from each set of the form $\pm f(\pm x)$.  Specifically, (remembering $ b^2 \neq ac$)   it suffices to consider only the quadratics $f$ in the set $\{ x^2 \pm 1, x^2+x,x^2+x-1\}$.

Now, if $f(x) = x^2-1$ then $(\beta, f(\beta)=\beta)$ is a primitive pair.
 
 If $f(x) = x^2+1$, then $(\beta-1, f(\beta-1)=-\beta)$ is a primitive pair.
 
 If $f(x) = x^2+x$, then $(\beta, f(\beta) =1-\beta)$  is a primitive pair.
 
 Finally, if $f(x) = x^2+x-1$, then$\beta, f(\beta)=-\beta)$ is a primitive pair.
\end{proof}

From now on assume $m \geq 3$.  Also assume that $m= 3^j m^\prime$, where $j$ is a non-negative integer and $\mathrm {gcd}(3,m^\prime)=1$. In fact, when $m'=1$ or $2$ we can assume $j$ is positive.
  For further computation, we need some additional results.   
 
   Furthermore, we consider the two cases 
 \begin{itemize}
 \item $m^\prime|q-1$
 \item $m^\prime \nmid q-1$
 \end{itemize}
 \textbf{Case A: $m^\prime| q-1$}
 \vspace{ .3cm}

 The following result is inspired from the \textbf{Lemma 6.1}, given by Cohen in \cite{11}.

\begin{lemma}\label{Lambdaeq}
 For $q=3^r$, where $r\geq 2$, let $d=q^m-1$ and let $g|x^m-1$ with $g_1, g_2, \ldots, g_k$ be the remaining distinct  irreducible polynomials dividing $x^m-1$
 Furthermore, let us write $ \Delta := 1 - \overset{k}{\underset{i=1}{\sum}} \frac{1}{q^{deg(g_i)}}$ and  $ \Lambda := \frac{k-1}{\Delta}+2$, with $\Delta>0$. 
 Let $m= m^\prime\, 3^j$, where $j$ is a non-negative integer and $\mathrm{gcd}(m^\prime,3)=1$. If $m^\prime| q-1$, then $$\Lambda = \frac{q^2-3q+aq+2}{aq-q+1} $$ 
 where $m^\prime= \frac{q-1}{a}$.  In particular, $\Lambda<q^2$.
\end{lemma}

 In order to apply our results, we also need the following lemma which can also be developed from \textbf{Lemma 6.2} by Cohen in \cite{11}. 

\begin{lemma}\label{Wbound}
For any positive integer $n$, $W(n)< 11.25\, n^{1/5}$, where $W$ has same meaning as stated earlier. 
\end{lemma}

Then taking $g=1$ in  inequality  (\ref{cond3})  and applying Lemma  \ref{Wbound}, we have the sufficient condition
$$ q^{\frac{m}{10}}>379.688\, q^2$$
Then for $m^\prime=q-1$, the inequality transforms to
 $$ q^{\frac{q-1}{10}-2}>379.688,$$
 which holds for $q\geq 81$.
\
Next, we consider $q=27$ and $m= m^\prime=q-1=26$. Then, by factorising,  $\omega(27^{26}-1)=12$ and  the pair $(q,m)=(27,26)$ satisfies the condition (\ref{cond3}).  Hence $\mathbb{F}_{27^{26}}$ contains a primitive normal element $\alpha$ such that $f(\alpha)$ is also primitive with given conditions.

 In order to reduce our  calculations, we now consider the range  $23\leq m^\prime < \frac{q-1}{2}$, for $q\geq 81$. Then, by Lemma {\ref{Lambdaeq} we have $\Lambda<\frac{q}{2}$. Hence the inequality (\ref{cond3}) is satisfied if $q^{\frac{m^\prime-1}{10}-1}> 189.844$ and this holds for $m^\prime\geq 23$. 
  
  When $m^\prime=\frac{q-1}{2}$, then $\Lambda\leq q$ and then the inequality is $q^{\frac{m^\prime-1}{10}-1}> 379.688$ and this holds for $m^\prime\geq 25$. Since $m^\prime=23 \neq \frac{q-1}{2}$ for any $q=3^r$ hence we leave this case.

 Next,  we  investigate all   cases with $m^\prime<23$. In the next part, we set $g=1$ unless mentioned otherwise.
\begin{itemize}
\item\textbf{Case 1: $m^\prime=1$}

 Then $m= 3^j$, $j$ is a positive integer.
To check the condition we take $g=1$. In that case $\Delta=\frac{2}{3}$ and $\Lambda=2$.
Then the inequality becomes $$q^{\frac{3^j}{10}}> 759.375.$$

Taking $q=3$, we have that the condition holds for  $j \geq 4$.  If it does not hold, then $q^m\leq 3^{61}$. So we calculate the rest of the pairs $(q,m)$ by calculating $\omega=\omega(q^m-1)$ i.e., the number of distinct prime divisors of $q^m-1$. Hence it suffices to check that $q^{m/2}> 3 W(q^m-1)^2$, where $W(q^m-1)=2^{\omega}$.

 After calculation we have that the following pairs are possible exceptional pairs. 

\begin{center}
\begin{tabular}{|c|}
\hline 
$(3,3),\quad (3,9),\quad (3,27), \quad (9,3),\quad(27,3),  \quad (81,3)$ \\ 
\hline 
\end{tabular} 
\end{center}

\item\textbf{Case 2: $m^\prime=2$}

 In this case, $m$ is of the form $m=2.3^j$, where $j$ is a positive integer.
Then $x^{m^\prime}-1= x^2-1$ splits into two distinct linear polynomials. We take  $g=1$ and then calculate the following.
  For $q=3$, $\Lambda\leq 5$ and the sufficient condition is $ q^{\frac{2.3^j}{10}}> 1898.44$ , which holds for $j\geq 4$.
 For $q=9$ and $q=27$, $\Lambda< 3.3$, and the condition is $ q^{\frac{2.3^j}{10}}> 1253$, which holds for $j\geq 3$. 
  Again for $3^4\leq q \leq 3^{10}$, $\Lambda<3.026$ and we need to check
 $ q^{\frac{2.3^j}{10}}> 1148.93$, which holds when $j\geq 2$. 
  For $3^{11}\leq q \leq 3^{32}$, $\Lambda< 3.0001$ and the condition is $ q^{\frac{2.3^j}{10}}> 1139.07$, which holds for $j\geq 1$; and for $q\geq 3^{33}$ the condition holds for $j\geq 0$.

   We calculate the remaining pairs by using $W(q^m-1)$ and $\Omega(x^2-1)$.  Then the condition is $q^{m/2}> 3 W(q^m-1)^2\Omega(x^2-1)\Lambda$. We obtain following pairs as possible exceptional pairs.

\begin{center}
\begin{tabular}{|c|}
\hline 
$ (3,6),\quad (3,18), \quad (9,6),\quad (27,6), \quad (81,6)$ \\
\hline 
\end{tabular} 
\end{center}

 From now on take $m=m'3^j$ with $j \geq0$.
\item\textbf{Case 3: $m^\prime=4$} 

 Here $m= 4.3^j$, with non-negative integer $j$. As there are $4$ distinct factors of $x^{m^\prime}-1$, so by calculation we have $\Delta>0$ if $q>3$. Then $\Lambda\leq 7.4$ for $q= 9$ and the sufficient condition is 
$q^{\frac{4.3^j}{10}}> 2809.69 $ which holds for $j\geq 3$.

 For $27\leq q \leq 243$, $\Lambda\leq 5.16$, and sufficient condition is $q^{\frac{4.3^j}{10}}> 1959.19 $. This holds when $j\geq 2$. Again, $729\leq q \leq 3^{17}$ the condition is $q^{\frac{4.3^j}{10}}> 1904.89 $ and holds for $j\geq 1$. When $q\geq 3^{18}$, the condition holds for $j\geq 0$.

 Taking $g=1$, we check the remaining pairs for the inequality $q^{m/2}>3.2^{2\omega}\Lambda$ and have the following pair as possible exceptional pair which does not  satisfy the inequality. 
\begin{center}
\begin{tabular}{|c|}
\hline 
$(9,4) $\\
\hline 
\end{tabular} 
\end{center}

  Similarly, proceeding in similar manner, we have the following pair as the solitary  possible exceptional pair from the remaining cases   in which  $m'\leq22$.
\begin{center}
\begin{tabular}{|c|}
\hline 
$(9,8) $\\
\hline 
\end{tabular} 
\end{center}

\end{itemize}

For each of the individual  pairs $(q,m)$ listed above that do not satisfy the sufficient condition based on Lemma \ref{Wbound}, we can test them more precisely by means of the sufficient condition (\ref{cond3}) after factorising completely $x^m-1$ and $q^m-1$ and making a choice of polynomial divisor $g$ of $x^m-1$ and factor $d$ of $q^m-1$.  In practice, the best choice is to choose $p_1, \ldots, p_n$ and sometimes, the ``largest'' irreducible factors $g_1, \ldots, g_k$ of $x^m-1$ to ensure that $\Delta$ is positive (and not too small).   Here the multiplicative aspect of the sieve is more significant.   Table 1 summarises the pairs in which the test yielded a positive conclusion: in only one case was $g \neq x^m-1$.

\vspace{.5cm}

\begin{center}
\begin{tabular}{|c|c|c|c|c|c|c|c|}
\hline 
$(q,m)$ & $d$ & $n$ & $g$ & $k$ & $\Lambda$ &$q^{m/2}$ & $3 W(d)^2\Omega(g)\Lambda$ \\ 
\hline 
(3,18) & 14 & 4 & $x^{18}-1$ & 0 & 12.231 & 19683 & 2348.65 \\ 
\hline 
(3,27) & 26 & 4 & $x^{27}-1$ & 0 & 9.18577 & $2.76145\times 10^6$& 881.834 \\  
\hline 
(9,5) & 2 & 2 & $x-1$ & 0 & 7.0939 & 243 & 85.1275 \\ 
\hline 
(9,7) & 1094 & 1 & $x^{7}-1$ & 0 & 3.01 & 2187  & 577.92 \\ 
\hline 
(9,8) & 10 & 3 & $x^2+1$ & 4 & 19.1006 & 6561 &   916.803 \\ 
\hline 
(9,9) & 14 & 4 & $x^9-1$ & 0 & 12.231 &  19683  &782.784 \\ 
\hline 
(27,5) & 22 & 2 & $x^{5}-1$ & 0 & 5.54729 & 3788  & 1065.08 \\ 
\hline 
(27,8) & 10 & 5 & $x^{8}-1$ & 0 & 20.5968 & 531441  & 31636.7 \\ 
\hline

\end{tabular}

\end{center}

\begin{center}
Table 1
\end{center}

 From the above table and calculation, we conclude that the following pairs are the final possible exceptional pairs, where $m^\prime|q-1$ and $(q,m)$ does not  satisfy the sufficient condition.
\begin{center}
\begin{tabular}{|c|}
\hline 
$((3,3), \quad3, 6) \quad (3,9)\quad  (9,3)\quad (9,4))$\\
$(9,6),\quad (27,3),\quad (27,4),\quad(81,3)$\\ 
\hline 
\end{tabular} 
\end{center}

\bigskip
 \textbf{Case B: $m^\prime\nmid q-1$}

 Let $u$ be the order of $q$ mod $m^\prime$. Then $x^{m^\prime}-1$ is a product of irreducible polynomial factors of degree less than or equal to $u$ in $\mathbb{F}_q[x]$.
In particular, $u\geq 2$ if $m^\prime \nmid q-1$.

 Let $M$ be the number of distinct  irreducible polynomials of $x^m-1$ over $\mathbb{F}_q$ of degree less than $u$. 
 
  Let $\vartheta(q,m)$ denotes the ratio
 $$ \vartheta(q,m):= \frac{M}{m},$$

   where $$m\vartheta(q,m)= m^\prime\vartheta(q,m^\prime).$$ 

 From \textbf{Proposition 5.3} in \cite{3}, we deduce the following bounds.

\begin{lemma}\label{rhobound}
Suppose $q=3^r$. Then the  following hold.
\begin{itemize}
\item $\vartheta(q,m)\leq \frac{1}{2}$, for $m=2 \,\mathrm{gcd}(q-1, m^\prime)$,
\item $\vartheta(q,m)\leq \frac{3}{8}$, for $m= 4\, \mathrm {gcd}(q-1, m^\prime)$,
\item $\vartheta(q,m)\leq \frac{1}{3}$, otherwise. 

\end{itemize}
\end{lemma}

 Now, to discuss the conditions, we need the following lemma, which is inspired by  \textbf{Lemma 7.2} in \cite{11}. We use almost the same procedure as in \cite{11}: hence the  proof is omitted.

\begin{lemma}
Assume that $q=3^r$ and $m$ is a positive integer such that $m^\prime\nmid q-1$.  Let $u (>1)$ denote the order of $q \mod m'$.   Let $g$ be the product of the irreducible factors of $x^{m'}-1$ of degree less than $u$.  Then, in the notation of Lemma  $\ref{Lambdaeq}$, we have $\Lambda\leq m^\prime$. 
\end{lemma}
We shalll break the discussion into 5 cases (I--V).  Case IV will treat pairs $(q,m)$ for which  $m=2 \,\mathrm{gcd}(q-1, m^\prime)$ and Case V those for which  $m=4 \,\mathrm{gcd}(q-1, m^\prime)$. 

For cases I--III we suppose, after  Lemma \ref{rhobound}, that   $\vartheta(q,m)\leq \frac{1}{3}$. In this situation let $g$ be the product of irreducible polynomials dividing $x^m-1$ of degree less than $u$. Then for $\mathfrak{N}(q^m-1,q^m-1,x^m-1)>0$
\\ It is sufficient to show
\begin{equation*}
q^{m/2}> 3m\, (11.25)^2 q^{2m/5}2^{m\vartheta(q,m)}:
\end{equation*}
hence, by Lemma \ref{rhobound}, whenever
\begin{equation*}
 q^{m/10} > 3 (11.25)^2 m 2^{m/3},
\end{equation*}
and so whenever
\begin{equation*}
 q^{m/10} > 380  m 2^{m/3}. 
\end{equation*}

 Now, $m^\prime \nmid q-1$, where $q=3^r$, $m>4$, then the inequality becomes 
\begin{equation} \label{cond4}
3^{rm/10}> 380 \, m\, 2^{m/3} .
\end{equation}

\textbf{Case I:  $q\geq 81 $} but $m\neq 2(q-1,m')$ or $4(q-1,m')$.

We apply ( \ref{cond4}) with $r=4$. 
This is satisfied if $m\geq 47$. So let $m<47$, in this case $q^m< 6.17 \times 10^{87}$ and $\omega\leq 49$. Applying this and Lemma \ref{rhobound} to the condition (\ref{cond3}), we obtain the sufficient  condition 
\begin{equation} \label{cond5}
 q^{m/2}> 3\,m\,W^2\,2^{m/3}, \quad  W=W(q^m-1),
 \end{equation} which  holds for $m\geq 37$.

 Next we consider $m\leq 36$ and $q^m< 5.07\times 10^{68}$. Then $\omega \leq 40$ and (\ref{cond5}) holds for $m \geq 31$.

 Proceeding this way we  conclude that (\ref{cond5}) holds for $m\geq 21$. For the remaining pairs we apply (\ref{cond5}) but with the precise factorization of $q^m-1$ an conclude that  for $q\geq 81$, when $m^\prime\nmid q-1$, for  all   pairs $(q,m)$,  $\mathbb{F}_{q^m}$ contains a primitive normal element $\alpha$ such that $f(\alpha)$ is also primitive.

\textbf{Case II:  $q=27, m\neq 4,8$}

Now apply (\ref{cond4}) with $r=3$. 
This  holds for $m\geq 108$. For next stage, we assume $m\leq 107$, then $q^m< 1.43 \times 10^{153}$ and applying it on the condition $ q^{m/2}> 3\,m\,W^2\,2^{m/3}$, we have the condition holds for $m\geq 79$.  Similarly, take $m\leq 78$, i.e., $q^m< 4.42\times 10^{111}$. Here $\omega\leq 57$ and the condition holds for $m\geq 60$. Next, 
assume $m\leq 59$: thus $q^m< 2.812\times 10^{84}$ with $\omega\leq 47$. Then (\ref{cond4}) holds  for $m\geq 48$. Similarly,  we find that (\ref{cond4}) holds for $m\geq 38$.

At this stage, as before, we test whether (\ref{cond5}) hold using the precise factorization of $q^m-1$.  The pairs will this test are
\begin{center}
\begin{tabular}{|c|}
\hline 
$(27,5),\quad(27,6),\quad(27,10)$ \\ 
\hline 
\end{tabular} 
\end{center}

\textbf{Case III:  $q=9, m\neq 16,32 $}

 For $m^\prime>4$ and $m^\prime\nmid q-1$, then  we apply these towards the condition $$ q^{m/2}> 3\,m\, W^2\,2^{m/3}.$$

 By calculating we have the following possible exceptional pairs which do not satisfy the condition. 
\begin{center}
\begin{tabular}{|c|}
\hline 
$(9,5)\quad (9,7) \quad (9,15) $ \\ 
\hline 
\end{tabular} 
\end{center}

 Finally we have  two additional cases.

 \textbf{Case IV: $m= 2\, \mathrm{gcd}(q-1,m^\prime).$}
 
 In this case $\vartheta(q,m)=\frac{1}{2}$ and proceeding as above we have the sufficient condition as $$q^{m/10}>380 m 2^{m/2}.$$

 For $q\geq 81$, the condition holds for $m\geq 116$. Next let $m<116$ so that $q^m < 2.9 \times 10^{219}$ and $\omega\leq 101$. Then we apply the condition 
 $$ q^{m/2}> 3m W^2 2^{m/2} $$.
 
This holds for $m\geq 80$. Hence we consider $m< 80$ so that $q^m< 5.89 \times 10^{150}$ and $\omega\leq 74$. For these the condition holds for $m\geq 60$. 

Following this order we obtained that the condition holds for $m\geq 34$ and we check the condition individually. Such pairs are $(81,32), (729,16), (6561,64)$ etc., and conclude  that all the pairs satisfy the sufficient condition, i.e., no exceptional pair in this case. 

 The remaining possible pairs are $(27,4), (27,12)$ etc. and by repeating the same process we have (27,4) as the only possible  exceptional pair.

\textbf{Case V: $m= 4\, \mathrm{ gcd}(q-1, m^\prime)$}

Here $\vartheta(q,m)=\frac{3}{8}$. We proceed as above and have the possible exceptional pair  

\begin{center}
\begin{tabular}{|c|}
\hline 
$(27,8)$ \\ 
\hline 
\end{tabular} 
\end{center}

Finally, from the possibly exceptional pairs $(q,m)$  already listed we  eliminate the following pairs by further calculation given in the table below by means of condition (\ref{cond3}).

\begin{center}
\begin{tabular}{|c|c|c|c|c|c|c|c|}
\hline 
$(q,m)$ & $d$ & $n$ & $g$ & $k$ & $\Lambda$ &$q^{m/2}$ & $3W(d)^2\Omega(g)\Lambda$ \\ 
\hline 
(9,5) & 22 & 1 & $x+2$ & 1 & 4.0682 &   243   & 195.274 \\ 
\hline 
(9,7) & 1094 & 1 & $x+2$ & 1 & 4.00367 &      2187      & 192.176 \\ 
\hline 
(9,15) & 14 & 6 & $(x+2)$ & 1 & 23.4645 &     $1.43487 \times 10^7$   & 1126.3 \\ 
\hline
 (27,5) & 22 & 2 & $x+2$ & 1 & 67.72974 &      3788.           &2167.35 \\ 
\hline 
(27,6) & 26 & 4 & $x+2$ & 1 & 17.5253 & 531441 & 841.214\\
\hline 
(27,8) & 10 & 5 & $x^{8}-1$ & 0 & 20.5968 & 531441  & 31636.7 \\ 
\hline 
(27,10) & 14 & 6 & $(x+1)(x+2)$ & 2 & 25.2471 & 14348907 & 4847.44\\
\hline
\end{tabular}
\end{center}

\begin{center}
Table 2
\end{center}

After all these calculations, we have the following pair for $m^\prime\nmid q-1$, which do not satisfy the sufficient condition for the existence of primitive normal element $\alpha$ in $\mathbb{F}_{q^m}$ such that $f(\alpha)$ is also primitive.
\begin{center}
\begin{tabular}{|c|}
\hline 
$ (27,4)$ \\ 
\hline 
\end{tabular} 
\end{center}
\hfill$\square$

 As the conclusion of all the cases considered, we have our final theorem.
\begin{theorem}\label{final}
Let $\mathbb{F}_{q^m}$ be a finite field of characteristic $3$. Then there exists a primitive normal element $\alpha$ in $\mathbb{F}_{q^m}$such that $f(\alpha)$ is also primitive in $\mathbb{F}_{q^m}$, where $f(x)=ax^2+bx+c$, with $a,b,c\in \mathbb{F}_{q^m}$ and $b^2\neq ac$, unless $(q,m)$ is one of the following pairs.
\begin{center}
\begin{tabular}{|c|}
\hline 
$(3, 6), \quad (3,9),\quad  (9,3), \quad (9,4), \quad(9,6)$\\ 
$(27,2),\quad (27,3),\quad (27,4), \quad(81,3)$\\
\hline 
\end{tabular} 
\end{center} 
\end{theorem}
\hfill$\square$
\begin{proof}
From the possible exceptions listed already it remains to exclude the pair $(3,3)$.  We do this by direct calculation through an elaboration of the method used to prove Lemma \ref{m2}.

An element $\alpha \in \mathbb{F}_{27}$ is normal over $ \mathbb{F}_3$ if and only if it has non- zero trace.  Define $\beta \in  \mathbb{F}_{27}$ such that $\beta^3=\beta^2-1$.  Then $\beta$ is a primitive normal element.   Moreover, $Tr(\beta)=Tr(\beta^2) =1$.  Calculate easily  the non-zero elements  $\alpha$ of $ \mathbb{F}_{27}$ as powers $j$ of $\beta$ and in the form $u\beta^2+v\beta+w$, where $u,v,w \in \{0,\pm1\}$.     Then, $\alpha$ is primitive if and only if gcd$(j,26)=1$: in particular if $\alpha$ is primitive then $-\alpha$ is not primitive. Further,  $\alpha$ is normal if and only if $u\neq v$.   From this working, there are nine  members of $ \mathbb{F}_{27}$ that are both primitive and normal, namely $\pm\beta,\pm \beta-1,  \beta^2 \pm1, -\beta^2-\beta \pm1, -\beta^2$.    Nine further members of $ \mathbb{F}_{27}$ are normal and three more are primitive, namely $-\beta^2+\beta+c$, where $c=0, \pm1$.  

We have to check that we can choose  a primitive normal element $\alpha$ such that $f(\alpha)$ is also primitive where $f(x)$ is any of the twelve quadratics $ \pm x^2\pm1, \pm x^2\pm x, \pm(x^2 \pm x -1)$.  This is easily verified even restricting $\alpha$ to the values $\beta,\beta-1, \-\beta^2$.
\end{proof}

It is likely that by means of computer calculation one can eliminate all the possible exceptional pairs listed in Theorem \ref{final} from being genuine exceptions.

\end{document}